\documentclass[12pt]{amsart}
\usepackage{fullpage}
\usepackage{amsmath,amsthm,amsfonts,amssymb,verbatim,eucal,amscd}
\usepackage[all]{xy}

\numberwithin{equation}{section}

\newtheorem{thm}[equation]{Theorem} 
\newtheorem{prop}[equation]{Proposition}
\newtheorem{lemma}[equation]{Lemma} 
\newtheorem{cor}[equation]{Corollary}
\newtheorem{example}[equation]{Example}
\newtheorem{question}[equation]{Question}

\theoremstyle{definition}
\newtheorem{remark}[equation]{Remark}
\newtheorem{definition}[equation]{Definition}

\DeclareMathOperator{\Tr}{Tr} 
\DeclareMathOperator{\Det}{Det}

\newcommand{\nat}{\ \natural \ }

\newcommand{\ot}{\otimes}

\newcommand{\eps}{{\varepsilon}}
\newcommand{\La}{{\Lambda}}

\title[Hopf automorphisms and twisted extensions]
{Hopf automorphisms and twisted extensions} 

\author{Susan Montgomery}
\address{Department of Mathematics, University of Southern California,
Los Angeles, CA}
\email{smontgom@usc.edu}

\author{Maria D.\ Vega}
\address{Department of Mathematics, North Carolina State University,
Raleigh, NC}
\email{mdvega@ncsu.edu} 

\author{Sarah Witherspoon}
\address{Department of Mathematics, Texas A\&M University, College Station,
TX}
\email{sjw@math.tamu.edu}

\date{June 5, 2015}

\thanks{The first author was supported by NSF grant DMS-1301860, the second author by an Alliance Post-Doctoral  Fellowship supported by  NSF grant DMS-0946431, and the third author by NSF grants DMS-1101399 and DMS-1401016.} 

\begin{document}

\maketitle

\begin{abstract}
We give some applications of a Hopf algebra constructed from
a group acting on another Hopf algebra $A$ as Hopf automorphisms, namely Molnar's smash coproduct
Hopf algebra.
We find connections between the exponent and Frobenius-Schur indicators
of a smash coproduct and the twisted exponents and twisted Frobenius-Schur indicators
of the original Hopf algebra $A$.
We study the category of modules of the smash coproduct.
\end{abstract}

%%%%%%%%%%%%%% Section 1 %%%%%%%%%%%%%%%%%

\section{Introduction}

Molnar \cite{Ml}  defined smash coproducts of Hopf algebras,
putting them on equal footing with the better-known smash products by
viewing both as  generalizations of semidirect products of groups.
%See also Radford  \cite{R2}. 
Recently smash coproducts have made an appearance as examples of new phenomena
in representation theory \cite{BW,DE}.
In this paper we propose several applications of smash coproducts.
In particular, the smash coproduct construction will allow us
to ``untwist'' some invariants defined via the action
of a Hopf algebra automorphism, such as the
twisted exponents and the twisted Frobenius-Schur indicators.

We note that considering Hopf automorphisms is a timely topic, since there has been
recent progress in determining the automorphism groups of some Hopf algebras 
\cite{AD, Ke, R3, SV, Y}. There has also been much recent work on indicators; their importance 
lies in the fact that they are invariants of the category of representations of the Hopf algebra, 
and may be defined for more abstract categories \cite{NSc}. Moreover the notion of twisted
indicators can be extended to pivotal categories \cite{SV3}.

We start by defining the smash coproduct $A\nat k^G$, for any Hopf algebra $A$
with an action of a finite group $G$ by Hopf automorphisms, in the next section.
In Section \ref{sec:exp} we recall the notions of exponent and
twisted exponent \cite{SV2} of a Hopf algebra, and 
find connections between the exponent
of  $A\nat k^G$ and twisted exponents of $A$ itself.
In Section \ref{sec:FS1} we assume the Hopf algebra $A$ is semisimple.
We recall definitions of Frobenius-Schur indicators \cite{KSZ} 
and twisted Frobenius-Schur indicators \cite{SV} for simple modules over the Hopf algebra, and 
give relationships between the 
indicators of the smash coproduct $A\nat k^G$ and twisted indicators of $A$ itself. 

In Section \ref{sec:FS2} we do not assume the Hopf algebra is semisimple. We introduce the twisted 
Frobenius-Schur
indicators of the regular representation of such a Hopf algebra, simultaneously generalizing
indicators for not necessarily semisimple Hopf algebras
\cite{KMN} and twisted indicators for semisimple Hopf algebras \cite{SV}. 
Again we find a connection with the Frobenius-Schur indicator of a smash
coproduct. We compute an example for which the Hopf algebra $A$ is of dimension 8
in Section \ref{sec:example}. 
Finally in Section \ref{sec:tp} we study the structure of categories of modules
of $A\nat k^G$, showing that they are equivalent to  semidirect product
tensor categories ${\mathcal{C}}\rtimes G$, where $\mathcal C$ is a
category of $A$-modules. 

Throughout, 
$k$ will be an algebraically closed field of characteristic 0.

%%%%%%%%%%%%%%% Section 2 %%%%%%%%%%%%%%%%%%%%

\section{The smash coproduct}\label{sec:prelim}

Our Hopf algebra was defined by Molnar \cite[Theorem 2.14]{Ml}, who called it the smash coproduct, although our definition 
seems different at first glance. See also \cite[p.\ 357]{R2}.

Let $A$ be a Hopf algebra over a field $k$ and let a finite
group $G$ act as Hopf algebra automorphisms of $A$. Let $k^G$ be the algebra 
of set functions from $G$ to $k$ under pointwise multiplication; that is, if $\{p_x \mid x\in G\}$ denotes the basis of $k^G$ dual to $G$,
then $p_x p_y  = \delta_{x,y} p_x$ for all $x,y\in G$. 
Recall that $k^G$ is a Hopf algebra with comultiplication given by  $\Delta(p_{x}) =  \sum _{y\in G} p_y \ot  p_{y^{-1}x},$ 
counit $\eps(p_x) = \delta_{1, x}$ and antipode $S(p_x) = p_{x^{-1}}$ for all $x\in G$. 

Then we may form the {\it smash coproduct Hopf algebra} 
$$K= A \nat k^G$$ 
with algebra structure the usual tensor product of algebras.
Denote by $a\nat p_x$ the element $a\ot p_x$ in $K$,
for each $a\in A$ and $x\in G$. 
Comultiplication is given by 
$$ \Delta(a \nat p_{x}) = \sum _{y\in G} (a_1 \nat p_y)\ot (( y^{-1}\cdot a_2)\nat p_{y^{-1}x})
$$
for all $x\in G$, $a\in A$. 
The counit and antipode are  determined by
$$
   \varepsilon(a \nat p_{x})= \delta_{1,x}\varepsilon(a)1 \ \ \mbox{ and } \ \ 
   S(a \nat p_{x})=  \ (x^{-1}\cdot S(a)) \nat p_{x^{-1}}.
$$
If $\La_A$ is an integral for $A$, then $\La_K = \La_A \nat p_1$ is an integral for $K$.

Note that Molnar defines the smash coproduct for the right coaction of any commutative 
Hopf algebra $H$. 
We show that our construction is actually his smash coproduct with $H =k^G$, by dualizing our $G$-action to a $k^G$-coaction. 

\begin{lemma} (1) $K$ as above is isomorphic to the smash coproduct as in \cite[Theorem~2.14]{Ml}, and thus is a Hopf algebra.

(2) If $A$ is finite-dimensional, then $K^* \cong A^* \# kG$, the smash product Hopf algebra as in \cite[Theorem 2.13]{Ml}.
\end{lemma} 

\begin{proof} (1) Given the left action of $G$ on $A$, we define $\rho: A \to A \ot k^G$ by 
$a \mapsto \sum_{x \in G}  (x \cdot a) \ot p_x$. 
Then $\rho$ is a right comodule map, using the fact that the $G$-action on $A$ 
satisfies  $x\cdot( y \cdot a) = (xy \cdot a)$ and $1\cdot a = a$ for all $x,y\in G$ 
and $a\in A$. 

Next we note that $A$ is a right comodule algebra under $\rho$ since the $G$-action is multiplicative, that is 
 $(x \cdot a) (x \cdot b) = x \cdot (ab)$. 
Also $A$ is a right comodule coalgebra, as the $G$-action preserves the coalgebra structure of $A$, that is, 
$x \cdot (\sum_a a_1 \ot a_2) = \sum_a (x \cdot a)_1 \ot (x \cdot a)_2.$ Thus $A$ is a right $k^G$-comodule bialgebra.

Finally the antipode also dualizes to the antipode given by Molnar, and thus Molnar's theorem \cite[Theorem 2.14]{Ml} applies.

(2) This is a special case of Molnar's result \cite[Theorem 5.4]{Ml}.

\end{proof}

%%%%%%%%%%%%%%%%%%% Section 3 %%%%%%%%%%%%%%%%%%%%%%%

\section{Hopf powers and exponents}\label{sec:exp}

In any Hopf algebra $H$, we denote the $n$th Hopf power of an element $x\in H$ by $x^{[n]} = \sum_x x_1 x_2 x_3 \ldots x_n;$ that is, 
first apply $\Delta_H$ $n-1$ times to $x$ and then multiply. Note that $x \mapsto x^{[n]}$ is a linear map. 

For $H$ semisimple, recall that the exponent of $H$, $exp(H)$, is the smallest positive integer $n$,
if it exists, such that 
$x^{[n]} = \eps(x) 1$ for all  $x \in H$. More generally, this definition makes sense whenever $S^2 = id$. We assume this 
property of $S$ unless stated otherwise.

Recently \cite{SV2} introduced the {\it twisted exponent}, 
where $exp$ is twisted by an automorphism of $H$ of finite order. 
Assume that $\tau \in Aut(H)$ and that $n$ is a multiple of the order of $\tau$. Define 
the $n$th {\em $\tau$-twisted Hopf power} of $x$ to be 
$$x^{[n,{\tau}]} := 
 \sum_x x_1 ({\tau}\cdot x_2) ( {\tau^2}\cdot x_3)  \ldots ({\tau^{n-1}}\cdot x_n).$$ 
\begin{definition} \label{expdef}{\rm $exp_{\tau}(H)$ is the smallest positive integer $n$, if it exists, such that $n$ is a
multiple of the order of $\tau$ and 
$x^{[n,{\tau}]} = \eps(x) 1$ for all $x \in H$.}
\end{definition}
Since $\tau$ is a Hopf automorphism, $\eps(\tau \cdot x) = \eps(x)$ for any $x \in H$, and thus $\eps(x^{[n,{\tau}]}) = \eps(x^{[n]}) = \eps(x)$.
If $H$ is not semisimple and $S^2\neq id$ yet $S$ is still bijective,
there is a more general definition of the twisted exponent in \cite{SV2}. 

We will need the following proposition which is a special case of 
\cite[Proposition 3.4]{SV2}. 

\begin{prop} \label{SV2Prop}
Suppose that the Hopf automorphism $\tau$ of the semisimple Hopf algebra $H$
has order $r$, $exp_{\tau} (H)$ is finite, and $m$
is a positive integer. Then 
   $x^{[mr,\tau]} = \varepsilon(x) 1 $ for all $x\in H$ if and only if
$exp_{\tau}(H)$ divides $m$. 
\end{prop}

Next we give some formulas for our Hopf algebras $K = A\nat k^G$.

\begin{lemma} \label{twisted-powers}
Let $w = a \nat p_x \in A\nat k^G$, the smash coproduct as above. Then 
$$
(a \nat p_x)^{[n]} 
= \sum_{z \in G , \  z^n = x} \ a^{[n,{z^{-1}}]} \nat p_z .
$$
In particular for $w = \La_K = \La_A \nat p_1$, replace $z$ by $z^{-1}$. Then 
$$
\La_K^{[n]} 
= \sum_{z \in G , \  z^n = 1}
\La_A^{[n, z]} \nat p_{z^{-1}} .
$$
\end{lemma} 

\begin{proof} A calculation shows that 
$$ (a \nat p_x)^{[n]} 
= \sum_{z \in G , \  z^n = x} a_1(z^{-1} \cdot a_2) (z^{-2} \cdot a_3)\cdots  (z^{-(n-1)} \cdot a_n)  \nat p_z,
$$
which gives the first equation in the lemma. The second follows from the first. 
\end{proof} 

We now find a relation among the (twisted) exponents of $A$, $G$, and $K=A\nat k^G$. 

\begin{thm}
Assume that $S^2 = id$ in $A$. Then the exponent of $K$ is the least common multiple of $exp (G)$ and $exp_z(A)$ for all
$z\in G$.
\end{thm}

\begin{proof}
Let  $n = exp(K)$, so that 
$$(a \nat  p_x)^{[n]} = \eps(a  \nat p_x)1 = \eps(a) \delta_{x, 1} 1 = \eps(a) \delta_{x, 1} \sum_z p_z $$ 
for all $a\in A$ and $x \in G$. When $a = 1$, then $(p_x)^{[n]} = \delta_{x,1}1$ implies that $exp(G) = exp(k^G)$ divides $n$. 
Thus $z^n = 1$ for all $z \in G$. 
By the above calculation,  $(a \nat p_1)^{[n]} = \eps(a) 1$,
 and so by Lemma \ref{twisted-powers},   
$a^{[n,{z^{-1}}]} = \eps(a)$ for all $z \in G$ and $a \in A$.
Therefore by Proposition~\ref{SV2Prop}, $exp(K)$ is a common multiple of $exp(G)$
and $exp_z(A)$ for all $z\in G$.

Now let $m$ be any common multiple of $exp(G)$ and $exp_z(A)$ for all $z\in G$.
By Lemma~\ref{twisted-powers} and Proposition~\ref{SV2Prop}, 
\begin{eqnarray*}
   (a \nat p_x)^{[m]} & = & \sum_{z\in G, \ z^m =x} a^{[m,z^{-1}]} \nat p_z \\
    & =& \delta_{1,x} \sum_{z\in G} a^{[m,z^{-1}]}\nat p_z \\
    &=& \delta_{1,x} \varepsilon(a) \sum_{z\in G} p_z \ \ 
 = \ \  \varepsilon (a \nat p_x) 1_K .
\end{eqnarray*} 
Again by Proposition~\ref{SV2Prop}, $exp(K)$ divides $m$.
\end{proof}

\quad

We will use the following lemma in calculations.

\begin{lemma}\label{S-exp}
Let $H$ be a Hopf algebra 
and let $\tau$ be a Hopf automorphism
of $H$ whose order divides $n$. Then 
$ \ \ S (x^{[n,\tau]}) = \tau^{-1}\cdot ( S(x)^{[n,\tau^{-1}]}) $
for all $x\in H$. 
\end{lemma}

\begin{proof}
Since $S$ is an anti-algebra and anti-coalgebra map and $\tau^n=1$ by hypothesis,
\begin{eqnarray*}
  S\big(x^{[n,\tau]}\big) &=& S\left(\sum_x x_1 (\tau\cdot x_2) (\tau^2\cdot x_3)\cdots
    (\tau^{n-1}\cdot x_n ) \right) \\
   & = & \sum_x (\tau^{n-1} \cdot S(x_n)) (\tau^{n-2}\cdot S(x_{n-1})) \cdots (
   \tau^2\cdot S(x_3))(\tau\cdot S(x_2)) S(x_1)\\
  &=& \sum_x (\tau^{-1} \cdot S(x_n)) (\tau^{-2} \cdot S(x_{n-1})) \cdots 
  (\tau^{-1(n-2)} \cdot S(x_3)) (\tau^{-1 (n-1)} \cdot S(x_2)) S(x_1)\\
   &=& \tau^{-1}\cdot \left( \sum_x S(x_n) (\tau^{-1}\cdot S(x_{n-1})) \cdots (\tau^{-1(n-3)}
     \cdot S(x_3)) (\tau^{-1(n-2)} \cdot S(x_2)) (\tau^{-1(n-1)} \cdot S(x_1))\right)\\
   &=& \tau^{-1} \cdot \big(S(x)^{{[ n, \tau^{-1} ] }}\big) .
\end{eqnarray*}
\end{proof}

\begin{cor}\label{inverse} 
Let $H$ be a Hopf algebra for which $S^2=id$ and let $\tau$ be a Hopf automorphism of $H$.
Then 
$  exp_{\tau^{-1}}(H)= exp_{\tau}(H). $
\end{cor}

\begin{proof} It is clear from Lemma \ref{S-exp} that $x^{[n,\tau]} = \eps(x)1$ $\iff$ $S(x)^{[n,\tau^{-1}]} = \eps(x)1$ 
since $\tau$  and $S$ are bijective. Thus the two twisted exponents are the same.
\end{proof}

\begin{question} {\rm We ask if Corollary \ref{inverse} is true more generally.  That is, if the order of $\tau$ is $n$ and 
$m$ is relatively prime to $n$, then is $exp_{\tau^m}(H)= 
exp_{\tau}(H)? $ }
\end{question}

%%%%%%%%%%%%%%%%%%%% Section 4 %%%%%%%%%%%%%%%%%%%%%%%%%%%%%%%% 

\section{Modules and Frobenius-Schur indicators}\label{sec:FS1}

In this section, we assume $A$ is a semisimple Hopf algebra, and thus we may assume that $\La_A$ is a 
normalized integral, that is, $\eps(\La_A) = 1$. Then the integral $\La_K = \La_A \nat p_1$ of $K= A\nat k^G$ 
is a normalized integral of $K$. 

For any (left) $K$-module $M$, we may write 
$$  M = \bigoplus_{x\in G} M_x$$
where $M_x = p_x \cdot M$ is a $K$-submodule of $M$ for each $x\in G$. 
Note that each $M_x$ is also an $A$-module, by restricting the action to $A$. 

Let $\nu^K_m$ denote the $m$th {\em Frobenius-Schur indicator} for 
$K$-modules as in \cite{KSZ}, and let $\nu^A_{m,x}$ denote the $m$th {\em  twisted Frobenius-Schur
indicator} for $A$-modules, twisted by $x$, as in \cite{SV}. 
That is, if $V$ is a $K$-module with character (or trace function) $\chi_V$, then
$$
   \nu^K_m (V) = \chi_V (\La_K ^{[m]}).
$$
 If $W$ is an $A$-module with character $\chi_W$ and $x$ is an
automorphism of $A$ whose order divides $m$, then 
$$
   \nu_{m,x}^A (W) = \chi_W (\La_A^{[m,x]}).
$$ 
See \cite{SV} for general results on twisted indicators and for
computations of $\nu^A_{m,x}$ when $A = H_8$, the smallest semisimple
noncommutative, noncocommutative Hopf algebra. 

Our next theorem gives a relationship between the Frobenius-Schur indicators
of $K$ and the twisted Frobenius-Schur indicators of $A$. 

\begin{thm}\label{twisted-indicator1}  
For every $K$-module $M$,
$$
  \nu^K _m (M) = \sum_{x\in G, \ x^m=1} \nu^A_{m,x^{-1}}(M_x) .
$$
\end{thm}

\begin{proof}
Write $M= \oplus_{x\in G} M_x$ as before.
Then $\nu^K_m(M) = \sum_{x\in G} \nu^K_m(M_x),$ and we will now compute
$\nu^K_m(M_x)$ for an element $x$ of $G$, writing $\La = \La_A$ for
ease of notation: By Lemma~\ref{twisted-powers}, 
\begin{eqnarray*}
    \nu^K_m(M_x) & = & \chi_{M_x} (\La_K ^{[m]} ) \\
      &=& \chi_{M_x} \big(\sum_{z\in G, \ z^m=1} \La ^{[m,z]} \nat p_{z^{-1}}\big) \\
     &=& \delta_{x^m, 1} \chi_{M_x} (\La^{[m,x^{-1}]}) 
    \ \ = \ \  \delta_{x^m,1} \nu^A_{m,x^{-1}} (M_x).
\end{eqnarray*}
Summing over all elements of $G$, we obtain the stated formula.
\end{proof}

As a consequence, for example, if $x$ is an element of $G$ of order $n$ and
$M$ is a $K$-module for which $M= M_x$ (i.e.\ $M_y=0$ for all $y\neq x$),
then
$\nu^K_m(M) = 0$ for all $m<n$. 

In our next result, we show that a twisted Frobenius-Schur indicator 
may always be realized as a Frobenius-Schur indicator for a smash coproduct.
Let $\tau$ be any Hopf automorphism of $A$ of finite order $n$,
and let $G=\langle \tau\rangle $ be the cyclic subgroup of the automorphism group 
generated by $\tau$. Set $K= A\nat k^G$.

\begin{thm} For any $A$-module $N$, extend $N$ to be a $K$-module $M$ 
by 
letting $M_{\tau^{-1}} = N$ and $M_{x} = 0$ for all $x\in G$, $x\neq \tau^{-1}$. 
Then for every positive integer multiple $m$ of $n$, 
$$
  \nu^A_{m,\tau}(N) = \nu^K_{m} (M). 
$$
Thus every value of a twisted indicator for $A$ is the value of an ordinary indicator for 
a smash coproduct over $A$.  

\end{thm}

\begin{proof}
By Theorem~\ref{twisted-indicator1}, 
$$
   \nu^K_m(M)  =  \sum_{x\in G, \ x^m=1} \nu_{m,x^{-1}}^A (M_x) 
    = \nu^A_{m,\tau} (M_{\tau^{-1}})  =  \nu^A_{m,\tau}(N).$$
\end{proof} 

\begin{example} {\rm We illustrate the theorem using a non-trivial automorphism of $A = H_8,$ the Kac-Palyutkin 
algebra of dimension 8 which is neither commutative nor 
cocommutative. The Hopf automorphism group was found in \cite{SV}, Section 4.2. Let $A$ be generated by $x, y, z$ with the usual relations $x^2 = y^2 = 1$, $z^2 = {\frac{1}{2}}(1 + x+ y- xy)$, $xy=yx$, $xz = zy$ and $yz = zx$, where $x, y$ are group-like  and $\Delta (z) = {\frac{1}{2}}(1 \ot 1 + 1 \ot x+ y\ot 1- y\ot x)(z\ot z).$

Let $\tau= \tau_4$ be the automorphism of $A$ 
of order 2 that interchanges $x$ and $y$ and sends $z$ to 
${\frac{1}{2}}(-z+xz+ yz +xyz),$ and let $\chi$ be the character of the unique two-dimensional simple module $N$ of $A$. Then from \cite{SV}, 
$\nu_{2, \tau}^A(N) = -1$. 

Letting $G = \langle \tau \rangle$ and $K = A \nat k^G$,  $N$ becomes a $K$-module $M$ by setting $M_{\tau} = N$ and $M_1 = 0$. Then $\nu_2^K(M) = -1$.

}
\end{example}

%%%%%%%%%%%%%%% Section 5 %%%%%%%%%%%%%%%%%%%%%%%%

\section{ Frobenius-Schur indicators for non-semisimple Hopf algebras}\label{sec:FS2} 

Let $A$ be a finite-dimensional Hopf algebra that is not necessarily semisimple 
and for which $S^2$ is not necessarily the identity map. 
When $A$ is not semisimple, there does not exist a normalized integral, and so we 
cannot use the definition of indicator from the previous section. Instead we extend the work in \cite{KMN} 
and define twisted Frobenius-Schur indicators for $A$ itself 
and obtain connections to Frobenius-Schur indicators of smash coproducts. 
Fix $\tau$, a Hopf automorphism of $A$ whose order
divides the positive integer $m$. 
We define a variant of the 
$m$th twisted Hopf power map of $A$ to be  
$P_{m-1,\tau}: A \rightarrow A$,  given by 
$$
   P_{m-1,\tau} (a) = \sum_a (\tau^{m-1}\cdot a_1)(\tau^{m-2}\cdot a_2)\cdots
     (\tau^2\cdot a_{m-2}) (\tau\cdot a_{m-1})
$$ 
for all $a\in A$. 
We will use this map to define twisted Frobenius-Schur indicators, and then we 
will show how it relates to the twisted Hopf power maps defined in Section~\ref{sec:exp},
by giving equivalent definitions of twisted Frobenius-Schur indicators 
in Theorem~\ref{integral} and Corollary~\ref{RRLLformula}. 

The $m$th {\em twisted Frobenius-Schur indicator} of $A$ is 
$$
   \nu_{m,\tau}(A) : = \Tr ( S \circ P_{m-1, \tau}),
$$
the trace of the map $S\circ P_{m-1,\tau}$ from $A$ to $A$,
where $S$ is the antipode of $A$.

We choose this definition as it 
specializes to the definition of the Frobenius-Schur indicator of the regular representation $A$ 
for an arbitrary finite-dimensional Hopf algebra in \cite{KMN} when $\tau$ is the identity, and also to the definition of twisted Frobenius-Schur indicators in the semisimple case given in \cite[Theorem 5.1]{SV}. The 
indicator of the regular representation has also been considered in \cite{Sh}.

The following theorem generalizes part of \cite[Theorem~2.2]{KMN}. 

\begin{thm}\label{integral}
Let $\Lambda$ be a left integral of $A$ and $\lambda$ a right
integral of $A^*$ for which $\lambda(\Lambda)=1$.
Then
$$
   \nu_{m,\tau}(A) = \lambda ( S ( \La )^{[m,\tau]} ) . 
$$
\end{thm}

\begin{proof}
By \cite[Theorem~1]{R},  
\begin{eqnarray*}
   \Tr (S\circ P_{m-1,\tau}) & = & \sum \lambda ( S(\Lambda_2) S\circ P_{m-1,\tau}
     (\Lambda_1))\\
    & = & \sum \lambda ( S(\Lambda_m) S((\tau^{m-1}\cdot \Lambda_1) (\tau^{m-2}\cdot\Lambda_2)
     \cdots (\tau\cdot\Lambda_{m-1}) )\\
%   &= & \sum \lambda ( S( \La_m) S(\tau(\La_{m-1})) \cdots S(\tau^{m-1} (\La_1) ))\\
   &=& \sum \lambda ( S(\La_m) (\tau \cdot S(\La_{m-1})) \cdots (\tau^{m-1} \cdot S(\La_1)))\\
   &=& \sum \lambda (S(\La)_1 (\tau\cdot S(\La)_2) \cdots (\tau^{m-1} \cdot S(\La)_m) ) \ \
   = \ \  \lambda (S(\La )^{[m,\tau]} ) .
\end{eqnarray*}\end{proof}

A similar  proof to that of \cite[Corollary~2.6]{KMN} yields the following
result that will be useful for computations. 

\begin{cor}\label{RRLLformula}
Let $\La _r$ be a right integral of $A$ and $\lambda _r$ be a right integral
of $A^*$ for which $\lambda _r (\Lambda _r) =1$. Then 
$$
   \nu_{m,\tau} (A) = \lambda _r(\La _r^{[m,\tau]}).
$$
Similarly let $\Lambda _l$ be a left integral of $A$ and $\lambda _l$ be a left integral
of $A^*$ for which $\lambda _l(\Lambda _l) =1$. Then
$$
   \nu_{m,\tau}(A) = \lambda_l ( \tau^{-1} \cdot \Lambda _l^{[m,\tau^{-1}]}).
$$
\end{cor} 

\begin{proof}
The first statement follows immediately from Theorem~\ref{integral} 
and the fact that if
$\Lambda _l$ is a left integral, then $\Lambda_r := S(\Lambda _l)$ is a right integral,
and the value of $\lambda _r$ on each is the same. 

For the second statement, if $\lambda _r$ is a right integral, let $\lambda_l := \lambda _r\circ S$, a left
integral of $A^*$. Then again by Theorem~\ref{integral} and also Lemma~\ref{S-exp},
\begin{eqnarray*}
   \lambda_l (\tau^{-1} \cdot \Lambda _l^{[m,\tau^{-1}]}) & = & 
     \lambda _r(S(\tau^{-1} \cdot \Lambda _l^{[m,\tau^{-1}]})) \\
  &=& \lambda _r(\tau^{-1}\cdot (S(\Lambda _l^{[m,\tau^{-1}]} )))\\
     &=& \lambda_r (S(\La _l)^{[m,\tau]}) \ \ 
   = \ \  \lambda_r (\Lambda_r ^{[m,\tau]}).
\end{eqnarray*}
\end{proof} 

Now let $G$ be a group of Hopf algebra automorphisms of $A$, as in Section~\ref{sec:prelim}.
The next result is a connection between twisted indicators of $A$
and indicators of the smash coproduct $K= A\nat k^G$. 

\begin{thm}
$ \ \  \nu_m(K) = \displaystyle{\sum_{g\in G, \ g^m=1}
   \nu_{m,g} (A)} . $
\end{thm}

\begin{proof}
Note that $\Lambda_K = \Lambda \nat  p_1$ and 
$\lambda_{K^*} = \lambda \ot (\sum_{z\in G} z)$
(since e.g.\ $\varepsilon( z\cdot a) = \varepsilon(a)$).
By \cite[Theorem 2.2]{KMN} and our Lemmas~\ref{twisted-powers} and ~\ref{S-exp}, 
\begin{eqnarray*}
  \nu_m(K) & = & \lambda_{K^*} (S_K (\Lambda_K^{[m]} )) \\
   & = & \left(\lambda\ot \big(\sum_{z\in G} z\big)\right) 
       \left(S_K \big(\sum_{g\in G, \ g^m=1} 
        \Lambda^{[m,g]}   \ot p_{g^{-1}}\big)\right)\\
   & = & \left(\lambda\ot \big(\sum_{z\in G}z \big)\right) 
    \left(S_K\big(\sum_{g\in G, \ g^m=1}
   \Lambda_1 (g\cdot\Lambda_2) \cdots (g^{m-1}\cdot\Lambda_m)\big) \ot  p_{g^{-1}}\right)\\
  & = & \sum_{g\in G, \ g^m=1} \lambda ( g\cdot S(\Lambda_1 (g\cdot \Lambda_2)\cdots
    (g^{m-1} \cdot\Lambda_m)))\\
    & = & \sum_{g\in G, \ g^m=1} \lambda ( S(\Lambda)^{[m,g^{-1}]} )\\
   & = & \sum_{g\in G, \ g^m=1} \nu_{m, g^{-1}} (A) \ \ 
    = \ \ \sum_{g\in G, \ g^m=1} \nu_{m, g } (A) .
\end{eqnarray*} 
\end{proof} 
   
In the next section, 
 we compute an example, a non-semisimple Hopf algebra of dimension~8 and its Hopf automorphism group. 
 
 %%%%%%%%%%%%%%%%% Section 6 %%%%%%%%%%%%%%%%%%%%%%%%%%%%%%%%%

\section{A non-semisimple example}\label{sec:example}

Let $A$ be the Hopf algebra defined as  
$$
A=k\langle g, x, y \ | \ gx=-xg, \ gy=-yg, \ xy=-yx, \ g^2=1, \ x^2=y^2=0 \rangle
$$
 with coalgebra structure  given by:
$$\Delta(g)= g \ot g, \ \varepsilon(g)=1, \  S(g)=g,$$
$$
\Delta(x)= x \ot g + 1 \ot x, \ \varepsilon(x)=0,  \ S(x)=gx,
$$
$$
\Delta(y)= y \ot g +1 \ot y, \ \varepsilon(y)=0, \  S(y)=gy.
$$
The element  $\Lambda= xy +xyg$ is both a right and left integral for $A$, and $\lambda= (xy)^*$ is both a 
right and left integral for $A^*$ such that $\lambda(\Lambda)=1$.

\begin{lemma} Let $V$ be the $k$-span of $x$ and $y$. Then $Aut(A) \cong Gl_2(V)$.
\end{lemma} 

\begin{proof} This is close to the examples considered in \cite{AD}, as $A$ is pointed and generated by its group-like and skew-primitive elements. However we provide an elementary proof for completeness. 

The coradical of $A$ is given by $A_0 = k \langle g \rangle$. Any automorphism $\tau$ of $A$ stabilizes $A_0$ 
and so fixes $g$. The next term of the coradical filtration is 
$$A_1 = A_0 \oplus V \oplus gV,$$ 
since $V$ is the set of $(g, 1)$-primitives and $gV$ is the set of $(1, g)$-primitives. Consequently $V$ and $gV$ 
are each stable under the action of $\tau$. But the $\tau$-action on $V$ determines the $\tau$-action on $gV$,
and also on $A = A_1 \oplus W$, where $W$ is the span of $xy$ and $gxy$. 

Conversely it is easy to check that any invertible linear action on $V$ preserves all of the relations of $A$, and thus gives an automorphism. 
\end{proof}

For an automorphism $\tau$ of order 2 or 3, we are able to compute some values of the indicators, 
using Corollary~\ref{RRLLformula}.  
We identify $\tau$ with a matrix 
$$\begin{pmatrix}
x \\
y  \end{pmatrix}
\stackrel{\tau}{\longmapsto} \begin{pmatrix}
  a& b \\
 c & d   \end{pmatrix}  \begin{pmatrix}
x \\
y  \end{pmatrix},$$
where $a,b,c, d \in k$, such that $\Det(\tau) = ad-bc\neq0$.

\begin{prop}  Case (1). If $\tau^2 = 1$ and $m$ is even, 
then 
$\nu_{m,\tau}(A)=\frac{m^2}{2} \big( 1 + \Det (\tau) \big).$
Consequently,  
$$
%\label{nu2tau}
\nu_{m,\tau}(A) = \begin{cases} m^2, &\mbox{if } \Det(\tau)=1\\ 
0, & \mbox{if } \Det(\tau)=-1 . \end{cases} 
$$

Case (2). If $\tau^3=1$, then 
$\nu_{3,\tau}(A)=
 \left( \Tr(\tau) + \Det(\tau)\right)^2 + \left(\Tr(\tau) + 1\right)
   \left(1 - \Det(\tau)\right).   $
Consequently,  
$$
%\label{nu4tau}
\nu_{3,\tau}(A) = \begin{cases} 9, &\mbox{if } \tau = id \\ 
0, & \mbox{if } \tau\neq   id .\end{cases} 
$$
 \end{prop} 
 
\begin{proof} We verify the formulas by using the first part of  
Corollary~\ref{RRLLformula}. 

Case (1): Recall that $\Lambda = xy + xyg$ and $\lambda = (xy)^*$ are right
integrals. We must find $\lambda(\Lambda^{[m,\tau]})$. 
First we will show that $\lambda ( (xy)^{[m,\tau]}) = \frac{m^2}{4}
(1 + \Det(\tau))$, and then we will argue that $\lambda ((xyg)^{[m,\tau]})
= \lambda ((xy)^{[m,\tau]})$.
In order to find $(xy)^{[m,\tau]}$, first note that
\begin{eqnarray*}
\Delta^{m-1}(x) & = & x\ot g^{\ot^{m-1}}  + 1\ot x \ot g^{\ot^{m-2}}
     + \cdots + 1^{\ot^{i-1}}\ot x \ot g^{\ot^{m-i}} +  \cdots +  1^{\ot^{m-1}}\ot x , \\
   \Delta^{m-1}(y) & = & y \ot g^{\ot^{m-1}}  + 1\ot y \ot g^{\ot^{m-2}}
     + \cdots + 1^{\ot^{i-1}}\ot y \ot g^{\ot^{m-i}} +  \cdots +  1^{\ot^{m-1}}\ot y ,
\end{eqnarray*}
each sum consisting of $m$ terms. Set
$$
   x_1 =  x\ot g^{\ot^{m-1}} , \ldots , \ x_i = 1^{\ot^{i-1}}\ot x \ot g^{\ot^{m-i}} , 
  \ldots, \ x_m = 1^{\ot^{m-1}}\ot x ,
$$
the index indicating the position of $x$ in the tensor product, and similarly
define $y_1,y_2,\ldots,y_m$.
Letting $\mu$ denote the multiplication map, by definition we have  
$$
   (xy)^{[m,\tau]} = \mu ( (1\ot \tau)^{\ot^{\frac{m}{2}}} )
    \left(\sum_{i,j=1}^m x_i y_j\right) . 
$$
Since $\tau\cdot g = g$ and $\lambda =(xy)^*$, in computing $\lambda ( (xy)^{[m,\tau]})$,
the only terms in the above expansion of $(xy)^{[m,\tau]}$ yielding a nonzero
value of $\lambda$ are those with an even number of factors of $g$.
These are precisely the terms $x_iy_j$ for which $i,j$ have the same parity,
of which there are $\frac{m^2}{2}$ terms. 
If $i,j$ are both odd (of which there are $\frac{m^2}{4}$ pairs), then
in $(xy)^{[m,\tau]}$, the $(i,j)$ term is simply $xy$ by the following observations:
(1) $\tau$ is applied only to factors of $g$ or $1$, which are fixed by $\tau$,
(2) if $i\leq j$, there are an even number of factors of $g$ between $x$ and $y$ after
applying $\mu$, and  (3) if $i>j$, there are an odd number of factors of $g$ between
$x$ and $y$ after applying $\mu$ (since $x_i$ is to the left of $y_j$), so moving
factors of $g$ to the right, past $x$, results in a factor of $(-1)$, and then applying
the relation $yx = -xy$ results in another factor of $(-1)$, so that the end result
is a term $xy$. 
If $i,j$ are both even (of which there are $\frac{m^2}{4}$ pairs),
then in $(xy)^{[m,\tau]}$, the $(i,j)$ term is $\tau\cdot xy = \Det (\tau) xy$, by
similar reasoning. Therefore
$$
  \lambda ( (xy)^{[m,\tau]}) = \lambda \left(
    \frac{m^2}{4}  xy + \frac{m^2}{4} \Det (\tau) xy \right)
   = \frac{m^2}{4} \left( 1 + \Det(\tau)\right).
$$

Finally, in order to compute $\lambda ((xyg)^{[m,\tau]})$, note that we need only include
an extra factor of $g^{\ot m} $ on the right: 
$$
   (xyg)^{[m,\tau]} = \mu ((1\ot \tau)^{\ot^{\frac{m}{2}}})
   \left(\sum_{i,j=1}^m x_i y_j \right) (g^{\ot ^{m}}).
$$
Since $m$ is even, the number of new factors of $g$ to be included, in
comparison to our previous calculation, is even, and so a similar analysis
applies. One checks that the extra factors of $g$ do not affect the result,
and so 
$$
  \lambda ((xyg)^{[m,\tau]}) = \lambda ((xy)^{[m,\tau]}) = \frac{m^2}{4}
   \left(1 + \Det(\tau)\right).
$$
Consequently, $\nu_{m,\tau} (A) =\displaystyle{ \lambda \big(\Lambda^{[m,\tau]}\big) = 
\frac{m^2}{2} \left( 1 + \Det (\tau)\right)}$. 

To see the conclusion of Case (1), note that since $\tau^2=1$, the determinant of $\tau$
is either $1$ or $-1$.

Case (2): 
A similar analysis applies. Note that $\lambda( (xy)^{[3,\tau]} ) =
\mu(1\ot \tau\ot \tau^2) (\sum_{i,j=1}^3 x_iy_j)$ and that
$\tau^2(x)= (a^2 + bc)x + b(a+d)y$, $\tau^2(y) = c(a+d)x + (d^2+bc)y$.
In evaluating $\lambda((xy)^{[3,\tau]})$, we again need only consider
$(i,j)$ terms for which $i,j$ have the same parity. By contrast, in evaluating
$\lambda((xyg)^{[3,\tau]})$, we need only consider $(i,j)$ terms for
which $i,j$ have different parity. Thus we find
\begin{eqnarray*}
  \lambda ((xy)^{[3,\tau]}) & = & \lambda \left( xy + x(\tau^2\cdot y) + yg(\tau^2\cdot x) g
   +(\tau^2\cdot xy) + (\tau\cdot xy)\right)\\
   &=& 1 + (d^2 + bc) + (a^2 +bc) + (a^2+bc)(d^2+bc) - bc(a+d)^2 + (ad-bc),\\
\lambda((xyg)^{[3,\tau]}) &=& \lambda\left(xg^2(\tau\cdot y)g^4 + yg(\tau\cdot x)g^5
     + g (\tau\cdot x) g^2(\tau^2\cdot y)g + g(\tau\cdot y) g(\tau^2\cdot x) g^2 \right) \\
  &=& d + a + a(d^2+bc) - bc(a+d) -bc(a+d) + d(a^2+bc).
\end{eqnarray*}
Adding these together, we have
\begin{eqnarray*}
\lambda(\Lambda^{[3,\tau]}) & = & 1 + a+d+a^2+ad+d^2+a^2d+ad^2+a^2d^2
     -abc -2abcd - bcd + bc + b^2 c^2 \\
  &=& (\Tr(\tau) + \Det(\tau))^2 + (1+\Tr(\tau))(1-\Det(\tau)) .
\end{eqnarray*} 
 To see the conclusion in Case (2), one can check the possible Jordan forms of the 
matrix for $\tau$.
\end{proof}

%%%%%% Section 6 %%%%%%%%%%%

\section{Tensor products and category of modules}\label{sec:tp}

The following theorem generalizes \cite[Theorem 2.1]{BW} from the case that $A$ is 
a group algebra, to the case that $A$ is a Hopf algebra. Let $K =A\nat k^G$
as before, and recall that for $M$ a $K$-module and $x\in G$, $M_x$ denotes
$p_x\cdot M$, a $K$-submodule of $M$, and 
$M=\oplus_{x\in G} M_x$. 
If $y\in G$, define $ {}^y M_x$ to be $M_x$ as a vector space, with $A$-module structure
given by $a\cdot_y m = (y^{-1}\cdot a )\cdot m$ for all $a\in A$, $m\in M$. 

\begin{thm}
Let $M,N$ be $K$-modules. Then
\begin{itemize}
\item[(i)]
$\displaystyle{
   (M\ot N)_x \cong \bigoplus_{\substack{y,z\in G \\ yz=x}}
     M_y \ot \ {}^yN_z 
}$, and
\item[(ii)]
 $ (M^*)_x = \  ^{x}\! (M_{x^{-1}})^*$. 
\end{itemize}
\end{thm}

\begin{proof}
The proof is a straightforward generalization of that of 
\cite[Theorem 2.1]{BW}. We include details for completeness.
We will prove the statement for modules of the form
$M=M_y$, $N=N_z$.
Let $\phi: M_y\ot N_z\rightarrow  M_y\ot {}^y N_z$, where
the target module is a $K$-module on which $p_{yz}$ acts as
the identity and $p_w$ acts as 0 for $w\neq yz$, be defined by
$\phi(m\ot n) = m\ot n$ for all $m\in M_y$, $n\in N_z$.
We check that $\phi$ is a $K$-module homomorphism: 
Let $x\in G$, $a\in A$. Apply $\Delta$ to $a\nat p_x$ to obtain
$$
  \phi ( (a\nat p_x) (m\ot n)) =
    \sum \delta_{x,yz} \phi(a_1 m \ot (y^{-1}\cdot a_2) n).
$$
On the other hand, 
$$
  (a\nat p_x) \phi(m\ot n) = \sum \delta_{x,yz} a_1 m \ot (y^{-1}\cdot a_2)n.
$$
As $\phi$ is a bijection by its definition, it is an isomorphism of $K$-modules.

We will prove that since $M = M_y$, its dual satisfies 
$M^* = (M^*)_{y^{-1}}$, and that the corresponding 
underlying $A$-module structure on the vector space $(M^*)_{y^{-1}}$
is isomorphic to $^{y^{-1}}(M_y)^*$. To see this, first let $x\in G$, $f\in M^*$,
and $m\in M$. Then
$$
   ((1\nat p_x ) (f))(m) = f((1\nat p_{x^{-1}}) m) = \delta_{x^{-1},y} f(m).
$$
It follows that $(M_y)^* = M^* = (M^*)_{y^{-1}}$, as claimed.
The $A$-module structure on $(M^*)_{y^{-1}}$ may be determined by considering the
action on $M^*$ of all elements of $K$ of the form $a\nat p_{y^{-1}}$ where $a\in A$.
Let $f\in M^*$ and $m\in M$. Then
$$
  ((a\nat p_{y^{-1}})(f))(m) = f(S(a\nat p_{y^{-1}})m) =
   f((y\cdot S(a))m).
$$
Considering the restriction of $M^*=(M^*)_{y^{-1}}$ to an $A$-module in this way,
we see that the action of $a$ on the vector space $(M_y)^*$ is that of $a$ on
the $A$-module ${}^{y^{-1}}(M_y)^*$: 
$$
  (a\cdot_{y^{-1}} f)(m) = ((y\cdot a) f)(m) = f(S(y\cdot a)m) =
    f((y\cdot S(a))m).
$$
Therefore the $A$-module structure on the vector space $(M^*)_{y^{-1}}$ is that of the $A$-module
${}^{y^{-1}}(M_y)^*$. 
\end{proof}

\begin{remark} {\rm As a consequence of the theorem, the category of $K$-modules is 
equivalent to the {\em semidirect product tensor category}
${\mathcal {C}}\rtimes G$
where $\mathcal C$ is the category of $A$-modules.
By definition, ${\mathcal{C}}\rtimes G$ 
is the category $\oplus_{g\in G} {\mathcal{C}}$,
with objects $\oplus_{g\in G} (M_g, g)$ where each $M_g$ is an object of $\mathcal C$,
and tensor product $(M,g)\ot (N,h)= (M\ot {}^g N, gh)$. 
See \cite{T}, where the notation ${\mathcal{C}} [G]$ is used instead for this
semidirect product category.
For other occurrences of ${\mathcal{C}}\rtimes G$ in the literature,
see, for example, \cite{GNaNi,Ni}.}
\end{remark}


\begin{thebibliography}{J\c{S}}

\bibitem[AD]{AD} N. Andruskiewitsch and F. Dumas, On the automorphisms of $U^+_q(\frak g)$, In: Quantum groups, IRMA Lectures on Mathematical and Theoretical Physics, vol. \textbf{12}, pp. 107--133. European Mathematical Society, Zurich (2008).

\bibitem[BW]{BW}
D. Benson and S. Witherspoon, Examples of support varieties for Hopf algebras with non-commutative tensor products, Archiv der Mathematik \textbf{102} (2014), no.\ 6, 513--520. 

\bibitem[DE]{DE} S.\ Danz and K.\ Erdmann, Crossed products as twisted category algebras,
Algebr.\ Represent.\ Theor.\ DOI 10.1007/S10468-014-9493-8. 

\bibitem[EG]{EG}
P. Etingof and S. Gelaki, On the exponent of finite-dimensional Hopf algebras. Math. Res.
Lett., \textbf{6}(2):131--140, 1999.

\bibitem[GNaNi]{GNaNi} S.\ Gelaki, D.\ Naidu, and D.\ Nikshych,
Centers of graded fusion categories, Algebra Number Theory \textbf{3}
(2009), no.\ 8, 959--990. 

\bibitem[KMN]{KMN} Y. Kashina, S. Montgomery, and S. Ng, 
On the trace of the antipode and higher indicators, 
Israel J.\  Math.\ \textbf{188} (2012), 57--89. 


\bibitem[KSZ]{KSZ}
Y. Kashina, Y. Sommerhäuser, and Y. Zhu,  {\it On Higher Frobenius-Schur Indicators},  Mem. Amer. Math. Soc.  \textbf{181}  (2006),  no. 855, viii+65 pp. 

\bibitem[Ke]{Ke}
M. Keilberg,  Automorphisms of the doubles of purely non-abelian finite groups, Algebras and Representation Theory, to appear; arXiv:1311.0575.


\bibitem[LM]{LM}
V. Linchenko and S. Montgomery, A Frobenius-Schur theorem for Hopf algebras,
Algebr. Represent. Theory \textbf{3} (2000), 347--355.

\bibitem[Ml]{Ml} 
R. Molnar, Semi-direct products of Hopf algebras, J. Algebra \textbf{47} (1977), 29--51.

\bibitem[M]{Mo}
S. Montgomery, \emph{Hopf Algebras and Their Actions on Rings},  CBMS Lectures Vol. \textbf{82}, Amer. Math. Soc., Providence, 1997.


\bibitem[NSc]{NSc}  S.-H. Ng and P. Schauenburg, Higher Frobenius-Schur indicators 
for pivotal ategories,  {\it Hopf Algebras and Generalizations}, AMS Contemp. Math. 441, 
AMS, Providence, RI, 2007, 63--90. 


\bibitem[Ni]{Ni} D. Nikshych, Non-group-theoretical semisimple Hopf algebras from group actions
on fusion categories, Selecta Math.\ \textbf{14} (2008), no.\ 1, 145--161. 

\bibitem[R]{R} D. E. Radford, 
The group of automorphisms of a semisimple Hopf algebra over a field of
characteristic 0 is finite, 
Amer. J. Math. \textbf{112} (1990), 331--357. 

\bibitem[R2]{R2} D. E. Radford, 
\emph{Hopf Algebras}, World Scientific Publishing, 2012.


\bibitem[R3]{R3} D. E. Radford, On automorphisms of biproducts, arXiv:1503.00381. 


\bibitem[SV]{SV} D. Sage and M. Vega, Twisted Frobenius-Schur indicators for Hopf algebras, J. Algebra \textbf{354} (2012), 136--147. 

\bibitem[SV2]{SV2} D. Sage and M. Vega, Twisted exponents and twisted Frobenius-Schur indicators for Hopf algebras, 
Communications in Algebra, to appear; arXiv:1402.5201.

\bibitem[SV3]{SV3} D. Sage and M. Vega, Twisted Frobenius-Schur indicators for pivotal categories, in preparation.

\bibitem[Sh]{Sh} K. Shimizu, Some computations of Frobenius-Schur indicators of the 
regular representations of Hopf algebras, Algebr. Represent. Theory \textbf{15} (2012), 325--357.


\bibitem[T]{T} D. Tambara,
Invariants and semi-direct products for finite group actions on tensor categories,
J.\ Math.\ Soc.\ Japan \textbf{53} (2001), no.\ 2, 429--456. 

\bibitem[Y]{Y} M. Yakimov, Rigidity of quantum tori and the Andruskiewitsch-Dumas conjecture, Selecta Math 
(N.S.) \textbf{20} (2014), no. 2, 421--464.

\end{thebibliography}
\end{document}